\documentclass{amsart}

\usepackage{amsmath, amssymb, amsthm, amscd}

\theoremstyle{plain}
\newtheorem{theorem}{Theorem}[section]

\newtheorem{lemma}[theorem]{Lemma}

\newtheorem{conjecture}[theorem]{Conjecture}

\theoremstyle{definition}
\newtheorem{definition}[theorem]{Definition}

\theoremstyle{remark}
\newtheorem{remark}[theorem]{Remark}

\newcommand{\N}{\mathbb{N}}                   
\newcommand{\dl}{\delta}                      
\newcommand{\dlp}{\dl^+}                      
\newcommand{\dlm}{\dl^-}                      

\numberwithin{equation}{section}

\begin{document}

\title{A degree condition for cycles of maximum length in bipartite digraphs}
\author{Janusz Adamus}
\address{J.Adamus, Department of Mathematics, The University of Western Ontario, London, Ontario N6A 5B7 Canada,
         and Institute of Mathematics, Jagiellonian University, ul. {\L}ojasiewicza 6, 30-348 Krak{\'o}w, Poland}
\email{jadamus@uwo.ca}
\author{Lech Adamus}
\address{L. Adamus, Faculty of Applied Mathematics, AGH University of Science and Technology, al. Mickiewicza 30, 30-059 Krak{\'o}w, Poland}
\email{adamus@agh.edu.pl}
\thanks{The authors' research was partially supported by Natural Sciences and Engineering Research Council of Canada (J. Adamus) and Polish Ministry of Science and Higher Education (L. Adamus).}

\begin{abstract}
We prove a sharp Ore-type criterion for hamiltonicity of balanced bipartite digraphs:
For $a\geq2$, a bipartite digraph $D$ with colour classes of cardinalities $a$ is hamiltonian if $d^+(u)+d^-(v)\geq a+2$ whenever $u$ and $v$ lie in opposite colour classes and $uv\notin A(D)$.
\end{abstract}
\maketitle


\section{Introduction}
\label{sec:intro}

The main purpose of this note is to give a sharp Ore-type sufficient condition for hamiltonicity of balanced bipartite digraphs.
A \emph{digraph} $D$ is a pair $(V(D),A(D))$, where $V(D)$ is a finite set (of \emph{vertices}) and $A(D)$ is a set of ordered pairs of elements of $V(D)$, called \emph{arcs}. For vertices $u$ and $v$ from $V(D)$, we write $uv\in A(D)$ to say that $A(D)$ contains the ordered pair $(u,v)$.
For a vertex $v\in V(D)$, we denote by $d^+(v)$ (resp. $d^-(v)$) the number of vertices $u\in V(D)$ such that $vu\in A(D)$ (resp. $uv\in A(D)$). We call $d^+(v)$ and $d^-(v)$ the \emph{positive} and \emph{negative half-degree} of $v$, respectively.
Further, $\dlp(D)$ (resp. $\dlm(D)$) denotes the minimum of $d^+(v)$ (resp. $d^-(v)$) as $v$ runs over all vertices of $D$.
A digraph $D$ is \emph{bipartite} when $V(D)$ is a disjoint union of sets $X$ and $Y$ (the \emph{colour classes}) such that $A(D)\cap(X\times X)=\varnothing$ and $A(D)\cap (Y\times Y)=\varnothing$. It is called \emph{balanced} if $|X|=|Y|$.
See subsection~\ref{subsec:not} below for details on notation and terminology.

\begin{definition}
\label{def:Bk}
Consider a balanced bipartite digraph $D$ with colour classes $X$ and $Y$ of cardinalities $a$. For $k\geq0$, we will say that $D$ satisfies condition $A^*_k$ when
\[
d^+(u)+d^-(v)\geq a+k
\]
for all $u$ and $v$ from opposite colour classes such that $uv\notin A(D)$.
\end{definition}

Our main result is the following:

\begin{theorem}
\label{thm:main}
Let $D$ be a balanced bipartite digraph with colour classes $X$ and $Y$ of cardinalities $a$, where $a\geq2$.
If $D$ satisfies condition $A^*_2$, then $D$ contains an oriented cycle of length $2a$.
\end{theorem}

There are numerous sufficient conditions for existence of cycles in digraphs (see \cite{BT}). In this note, we will be concerned with the degree conditions.
For general digraphs, the Dirac- and Ore-type conditions are due, respectively, to Nash-Williams and Woodall.

\begin{theorem}[{Nash-Williams, \cite{NW}}]
\label{thm:NW}
Let $D$ be a digraph on $n$ vertices, where $n\geq3$. If $\dlp(D)\geq n/2$ and $\dlm(D)\geq n/2$, then $D$ contains an oriented cycle of length $n$.
\end{theorem}

\begin{theorem}[{Woodall, \cite{W}}]
\label{thm:W}
Let $D$ be a digraph on $n$ vertices, where $n\geq3$. If $d^+(x)+d^-(y)\geq n$ for every pair of distinct vertices $x,y\in V(D)$ satisfying $xy\notin A(D)$, then $D$ contains an oriented cycle of length $n$.
\end{theorem}

In terms of the total degrees, we have the following result of Meyniel (see \cite{BonT} for a short proof). Here $d(x)=d^+(x)+d^-(x)$.

\begin{theorem}[{Meyniel, \cite{Mey}}]
\label{thm:Mey}
Let $D$ be a digraph on $n$ vertices ($n\geq3$) in which, for any two distinct vertices $x$ and $y$, there is an oriented path from $x$ to $y$ and from $y$ to $x$. If $d(x)+d(y) \geq 2n-1$ for any two vertices $x$ and $y$ such that $xy\notin A(D)$ and $yx\notin A(D)$, then $D$ contains an oriented cycle of length $n$.
\end{theorem}

Naturally, for bipartite digraphs one can expect degree bounds of roughly $|D|/2$ rather than $|D|$.

\begin{theorem}[{Amar \& Manoussakis, \cite{AM}}]
\label{thm:AM}
Let $D$ be a bipartite digraph having colour classes $X$ and $Y$ such that $|X|=a\leq b=|Y|$.
If $\dlp(D)\geq (a+2)/2$ and $\dlm(D)\geq(a+2)/2$, then $D$ contains an oriented cycle of length $2a$.
\end{theorem}

In case $a=b$, the above theorem gives a Dirac-type condition for hamiltonicity of a balanced bipartite digraph. In \cite{AM}, one also finds a characterization of all the bipartite digraphs that do not contain an oriented cycle of length $2a$, but satisfy $\dlp(D)\geq(a+1)/2$ and $\dlm(D)\geq(a+1)/2$.

As far as the Ore-type conditions for bipartite digraphs go, relatively little is known. The following result of \cite{MM} was the main motivation for the present work.
A bipartite digraph $D$, with colour classes $X$ and $Y$ such that $|X|=a\leq b=|Y|$, is said to satisfy condition $A_k$ ($k\geq0$) when $d^+(u)+d^-(v)\geq a+k$ for all $u$ and $v$ such that $uv\notin A(D)$.

\begin{theorem}[{Manoussakis \& Milis, \cite{MM}}]
\label{thm:MM}
Let $D$ be a bipartite digraph with colour classes $X$ and $Y$ such that $|X|=a\leq b=|Y|$.
If $D$ satisfies $A_2$, then $D$ contains an oriented cycle of length $2a$.
\end{theorem}

The problem with the above result is that condition $A_2$ concerns \emph{all} pairs of non-neighbouring vertices of $D$. In particular, it concerns the pairs of vertices from the same colour class, which puts a very restrictive assumption on $D$. To make condition $A_2$ more meaningful, one thus needs to require that only the pairs of vertices from opposite colour classes be considered (as in Definition~\ref{def:Bk} above).

We conjecture the following (and prove it for $a=b$ in the next section).

\begin{conjecture}
\label{conj:main}
Let $D$ be a bipartite digraph with colour classes $X$ and $Y$ such that $|X|=a\leq b=|Y|$. If
\begin{equation}
\label{ab2}
d^+(u)+d^-(v)>\frac{a+b+2}{2}
\end{equation}
whenever $u$ and $v$ lie in opposite colour classes and $uv\notin A(D)$, then $D$ contains an oriented cycle of length $2a$.
\end{conjecture}

\begin{remark}
\label{rem:sharp}
We suspect that condition \eqref{ab2} is sharp, but we do not know how to generalize the following example of \cite{AM} (Fig.~1) for arbitrarily large $a$. Here $a=b=3$, and all the vertices have both positive and negative half-degree equal to $2$. Therefore, the sum of half-degrees of any pair of vertices is $4$; i.e., equals to $(a+b+2)/2$. However, no oriented cycle of length $6$ is contained in this digraph.
\end{remark}

\begin{center}
\setlength{\unitlength}{.8cm}
\begin{picture}(7,3.5)
\put(1,1){\circle{0.2}}
\put(3,1){\circle{0.2}}
\put(5,1){\circle{0.2}}
\put(1,3){\circle*{0.2}}
\put(3,3){\circle*{0.2}}
\put(5,3){\circle*{0.2}}
\put(1,3){\vector(0,-1){1.9}}
\put(1,3){\vector(2,-1){3.9}}
\put(3,3){\vector(0,-1){1.9}}
\put(3,3){\vector(1,-1){1.9}}
\put(5,3){\vector(-2,-1){3.9}}
\put(5,3){\vector(-1,-1){1.9}}
\put(1.1,1.1){\vector(1,1){1.8}}
\put(3,2){\vector(2,1){1.9}}
\put(2.9,1.1){\vector(-1,1){1.8}}
\put(4.5,2.5){\vector(1,1){0.4}}
\put(3,2){\vector(-2,1){1.9}}
\put(3.5,2.5){\vector(-1,1){0.4}}
\put(3,0){\makebox(0,0)[b]{\ Fig.~1}}
\end{picture}
\end{center}

\begin{remark}
\label{rem:ak}
Note also that the bound $(a+b+2)/2$ in \eqref{ab2} cannot be replaced, in general, by a bound of the type $a+k$, for any $k\in\N$.
Indeed, for $k\in\N$ and any $b\geq a+2k+2$, let $D$ be the disjoint union of digraphs  $K^*_{1,k+2}$ and $K^*_{a-1,b-k-2}$ (Fig.~2), where $K^*_{k,l}$ denotes the complete bipartite digraph with colour classes of cardinalities $k$ and $l$. Clearly $D$ does not contain an oriented cycle of length $2a$, but the sum of half-degrees of non-neighbouring vertices from opposite colour classes is either $(a-1)+(k+2)=a+k+1$ or $1+(b-k-2)=b-k-1$, so in any case it is greater than or equal to $a+k+1$.
\end{remark}

\begin{center}
\setlength{\unitlength}{.8cm}
\begin{picture}(8,5.2)
\put(2,1.5){\oval(2,1)}
\put(1.5,1){\makebox(1,1){$k+2$}}
\put(5.5,1.5){\oval(3,1)}
\put(4.5,1){\makebox(2,1){$b-k-2$}}
\put(2,4.1){\circle*{0.2}}
\put(5.5,4){\oval(2.4,1)}
\put(4.5,3.5){\makebox(2,1){$a-1$}}
\put(1.9,4){\vector(-1,-3){.65}}
\put(1.7,3.4){\vector(1,3){.2}}
\put(2,4){\vector(0,-1){1.9}}
\put(2,3.5){\vector(0,1){.5}}
\put(2.1,4){\vector(1,-3){.65}}
\put(2.3,3.4){\vector(-1,3){.2}}
\put(4.7,3.5){\vector(0,-1){1.5}}
\put(4.7,3){\vector(0,1){.5}}
\put(6.3,3.5){\vector(0,-1){1.5}}
\put(6.3,3){\vector(0,1){.5}}
\put(4.9,2.2){\vector(1,1){1.2}}
\put(5.1,2.4){\vector(-1,-1){.2}}
\put(6.1,2.2){\vector(-1,1){1.2}}
\put(5.9,2.4){\vector(1,-1){.2}}
\put(3.5,0){\makebox(0,0)[b]{\ Fig.~2}}
\end{picture}
\end{center}

\subsection{Notation and terminology}
\label{subsec:not}

This paper is concerned with digraphs, in the sense of \cite{BT}. That is, the set $A(D)$ of arcs of $D$ consists only of ordered pairs of vertices of $D$ (i.e., $D$ has no loops or multiple arcs). Given a digraph $D$, we denote by $V(D)$ the set of its vertices, and the number of vertices $|V(D)|$ is the \emph{order} of $D$.
We write $xy\in A(D)$ to say that an arc from a vertex $x$ to a vertex $y$ is contained in $D$. If $xy\in A(D)$, then $x$ and $y$ are called \emph{neighbours}. For a set $S\subset V(D)$, we denote by $N^+(S)$ the set of \emph{vertices dominated} by the vertices of $S$; i.e.,
\[
N^+(S)=\{v\in V(D): uv\in A(D)\text{\ for\ some\ }u\in S\}\,.
\]
Similarly, $N^-(S)$ denotes the set of \emph{vertices dominating} the vertices of $S$; i.e,
\[
N^-(S)=\{v\in V(D): vu\in A(D)\text{\ for\ some\ }u\in S\}\,.
\]
For $S=\{u\}$, we set $d^+(u)=|N^+(u)|$ and $d^-(u)=|N^-(u)|$, which we call the \emph{positive} and \emph{negative half-degree} of $u$, respectively\footnote{Also known in literature as the \emph{outdegree} and \emph{indegree}.}. Further, $\dlp(D)$ and $\dlm(D)$ denote respectively the least positive and the least negative half-degrees of $D$. A digraph obtained from $D$ by removing the vertices of $S$ and their incident arcs is denoted by $D\setminus V(S)$.

For $u\in V(D)$ and $S\subset V(D)$, we set $N^+_S(u)$ (resp. $N^-_S(u)$) to be the set of vertices of $S$ dominated by (resp. dominating) $u$, and denote its cardinality by $d^+_S(u)$ (resp. $d^-_S(u)$).

An oriented cycle (resp. oriented path) on $m$ vertices in $D$ is denoted by $C_m$ (resp. $P_m$). If the vertices are $v_1,\dots,v_m$, we write $C_m=[v_1,\ldots,v_m]$ and $P_m=(v_1,\ldots,v_m)$. We will refer to them as simply \emph{cycles} and \emph{paths} (skipping the term ``oriented''), since their non-oriented counterparts are not considered in this note at all.

Let $D$ be a bipartite digraph, with colour classes $X$ and $Y$. We say that $D$ is \emph{balanced} if $|X|=|Y|$. A \emph{matching from $X$ to $Y$} is an independent set of arcs with origin in $X$ and terminus in $Y$. If $G$ is balanced, one says that such a matching is \emph{complete} if it consists of precisely $|X|$ arcs. A path or cycle is said to be \emph{compatible} with a matching $M$ from $X$ to $Y$ if its arcs are alternately in $M$ and in $A(D)\setminus M$.

\medskip

\section{Proof of the main result}
\label{sec:proof}

In this section, we prove Theorem~\ref{thm:main}. For the rest of the paper, $D$ denotes a balanced bipartite digraph with colour classes $X$ and $Y$, where $|X|=|Y|=a$ (hence $|V(D)|=2a$). Recall condition $A^*_k$ of Definition~\ref{def:Bk}.

\subsection{Lemmas}
\label{subsec:lemmas}

The proof of Theorem~\ref{thm:main} is based on the following four simple lemmas and a remark.

\begin{lemma}
\label{lem:1}
If $D$ satisfies condition $A^*_0$, then $D$ contains a complete matching from $X$ to $Y$.
\end{lemma}

\begin{proof}
By the K{\"o}nig-Hall theorem (see, e.g., \cite{B}), it suffices to show that $|N^+(S)|\geq|S|$ for every set $S\subset X$. If $N^+(S)=Y$, then there is nothing to show. Otherwise, we can choose vertices $x\in S$ and $y\in Y\setminus N^+(S)$. Now $xy\notin A(D)$; therefore, by assumption,
\[
a\leq d^+(x)+d^-(y)\leq|N^+(S)|+|X\setminus S|=|N^+(S)|+a-|S|\,.
\]
Hence $|N^+(S)|\geq|S|$, as required.
\end{proof}

\begin{remark}
\label{rem:1}
Suppose $D$ contains a complete matching $M$ from $X$ to $Y$, and let $(p_1,\ldots,p_s)$ be a path in $D$ compatible with $M$, and of maximal length among paths compatible with $M$. (We will say \emph{``maximal path compatible with $M$''} for short.) Denote this path by $P$.
It follows from maximality of $P$ that $p_1\in X$ and $p_s\in Y$. Hence, in particular, $s$ is even.

Indeed, if $p_1\in Y$, then $p_1$ is dominated by a vertex $x\in X\setminus V(P)$ such that $xp_1\in M$ (by completeness of $M$). If $x=p_s$, then $P$ is, in fact, a cycle and we can renumber its vertices so that $p_1\in X$ (and hence $p_s\in Y$). Otherwise, $(x,p_1,\ldots,p_s)$ is a path compatible with $M$ of length greater than $P$; a contradiction. Similarly, if $p_s\in X$ (and $p_sp_1\notin M$) then there exists $y\in Y\setminus V(P)$ such that $p_sy\in M$, again contradicting the maximality of $P$.
\end{remark}

\begin{lemma}
\label{lem:2}
Assume that $D$ satisfies condition $A^*_1$, and the order of $D$ is at least $4$ (i.e., $a\geq2$).
Choose $M$ a complete matching from $X$ to $Y$ and $P$ a maximal path compatible with $M$. Write $P=(p_1,\ldots,p_s)$.
If $p_sp_1\in A(D)$, then $D$ contains an oriented cycle $C_{2a}$ compatible with $M$.
\end{lemma}

\begin{proof}
We will show that $s=2a$. For a proof by contradiction, suppose otherwise, so $Y\setminus V(P)\neq\varnothing$.

If $yp_i\in A(D)$ for some $y\in Y\setminus V(P)$ and $p_i\in V(P)$, then
\[
(y,p_i,p_{i+1},\ldots,p_s,p_1,\ldots,p_{i-1})
\]
is a path compatible with $M$ and longer than $P$; a contradiction.
We can thus assume that no vertex of $P$ is dominated by a vertex from $Y\setminus V(P)$. Hence $d^-(p_i)\leq |V(P)|/2=s/2$ for all $p_i\in V(P)$, and $d^+(y)\leq|X\setminus V(P)|=a-s/2$ for all $y\in Y\setminus V(P)$. Therefore, for any $p_i\in X\cap V(P)$ and $y\in Y\setminus V(P)$, we have
\[
a+1\leq d^+(y)+d^-(p_i)\leq (a-s/2)+s/2=a\,.
\]
The contradiction proves that $Y\setminus V(P)=\varnothing$, and hence $s=2a$.
\end{proof}

\begin{lemma}
\label{lem:3}
Assume that $D$ satisfies condition $A^*_k$, where $k\geq1$, and the order of $D$ is at least $4$ (i.e., $a\geq2$). If $M$ is a complete matching from $X$ to $Y$,
then there esists $l$, $l\geq a+k$, such that $D$ contains an oriented cycle $C_l$ compatible with $M$.
\end{lemma}

\begin{proof}
Let $P$ be a maximal path compatible with $M$. Write $P=(p_1,\ldots,p_s)$.
If $p_sp_1\in A(D)$, then, by Lemma~\ref{lem:2}, $D$ contains a cycle $C_{2a}$ compatible with $M$. Suppose then that $p_sp_1\notin A(D)$.
Recall that $p_1\in X$ and $p_s\in Y$ (Remark~\ref{rem:1}). By maximality of $P$, vertex $p_1$ is not dominated by any $y\in Y\setminus V(P)$, and vertex $p_s$ does not dominate any $x\in X\setminus V(P)$. Therefore, by assumption,
\[
a+k\leq d^+(p_s)+d^-(p_1)=d^+_{V(P)}(p_s)+d^-_{V(P)}(p_1)\,,
\]
and hence $d^+_{V(P)}(p_s)\geq(a+k)/2$ or else $d^-_{V(P)}(p_1)\geq(a+k)/2$.

In the first case, let $i_0=\min\{i\colon p_sp_i\in A(D)\}$. Then $[p_{i_0},p_{i_0+1},\ldots,p_s]$ is a cycle in $D$ compatible with $M$ and of length at least $2d^+_{V(P)}(p_s)$, which is greater than or equal to $a+k$. In the second case, let $j_0=\max\{j\colon p_jp_1\in A(D)\}$. Then $[p_1,p_2,\ldots,p_{j_0}]$ is a required cycle of length at least $2d^-_{V(P)}(p_1)$, which is greater than or equal to $a+k$.
\end{proof}

\begin{lemma}
\label{lem:4}
Let $M$ be a complete matching from $X$ to $Y$ in $D$. Let $C$ be a maximal cycle in $D$ compatible with $M$, and let $(u_1,v_1,\ldots,u_p,v_p)$ be a path in $D\setminus V(C)$, denoted by $P$, compatible with $M$, where $u_i\in X$ and $v_i\in Y$.
If $d^-_{V(C)}(u_1)>0$ and $d^+_{V(C)}(v_p)>0$, then $d^+_{V(C)}(v_p)+d^-_{V(C)}(u_1)\leq m-p+1$, where $m$ is half the length of $C$.
\end{lemma}

\begin{proof}
Write $C=[x_1,y_1,\ldots,x_m,y_m]$, with $x_\nu\in X$ and $y_\nu\in Y$ ($1\leq\nu\leq m$). By assumption, there exist $y_i$ and $x_j$ on $C$ such that $y_iu_1\in A(D)$ and $v_px_j\in A(D)$. Let $(x_{i+1},y_{i+1},\ldots,x_{j-1},y_{j-1})$ be the path, denoted by $P^{ij}$, between $y_i$ and $x_j$ on $C$, traversed according to the orientation of $C$; of order, say, $2l$. Then $l\geq p$, because otherwise the cycle $[v_p,x_j,\ldots,y_i,u_1,v_1,\ldots,u_p]$ would be strictly longer than $C$.

We can choose the $y_i$ and $x_j$ so that $u_1$ is not dominated by any $y_\nu\in V(P^{ij})$, and that $v_p$ does not dominate any $x_\nu\in V(P^{ij})$.
Note that for every pair of vertices $y_s,x_{s+1}$ from $V(C)\setminus V(P^{ij})$ at most one of the arcs $y_su_1$ and $v_px_{s+1}$ belongs to $A(D)$, for else $D$ would contain a cycle
\[
[v_p,x_{s+1},y_{s+1},\ldots,x_s,y_s,u_1,v_1,\ldots,u_p]
\]
strictly longer than $C$. There is precisely $m-l-1$ of such pairs. Accounting for the arcs $y_iu_1$ and $v_px_j$, we get the required estimate
\[
d^+_{V(C)}(v_p)+d^-_{V(C)}(u_1)\leq (m-l-1)+2\leq m-p+1\,.
\]
\end{proof}

\subsection{Proof of Theorem~\ref{thm:main}}
\label{subsec:proof}

Assume then that $D$ satisfies condition $A^*_2$. Choose $M$ a complete matching from $X$ to $Y$, and an oriented cycle $C$, of length $2m$, compatible with $M$ in such a way that $C$ is of maximal length among all the oriented cycles in $D$ compatible with some complete matching from $X$ to $Y$. Write $C=[x_1,y_1,\ldots,x_m,y_m]$, with $x_\nu\in X$ and $y_\nu\in Y$, $1\leq\nu\leq m$.
By Lemma~\ref{lem:3}, $2m\geq a+2$.

We want to show that $m=a$. Suppose otherwise. Then we can choose a path $P$, of order $2p$, contained in $D\setminus V(C)$, compatible with $M$ and of maximal length among such paths in $D\setminus V(C)$. Write $P=(u_1,v_1,\ldots,u_p,v_p)$, with $u_\nu\in X$ and $v_\nu\in Y$, $1\leq\nu\leq p$ (cf. Remark~\ref{rem:1}). Let $R$ denote the remaining vertices of $D$; i.e., $R=V(D)\setminus(V(C)\cup V(P))$. Write $|R|=2r$ for some $r\geq0$. Then
\[
a=m+p+r\quad\mathrm{and}\quad 2p+2r=2a-2m\leq a-2\,.
\]
The remainder of the proof splits into several cases according to the properties of $d^-_{V(C)}(u_1)$ and $d^+_{V(C)}(v_p)$. Note that, by maximality of $P$, we have $d^-_{V(R)}(u_1)=0$ and $d^+_{V(R)}(v_p)=0$.

\subsection*{Case A: $d^-_{V(C)}(u_1)=0$}

\subsection*{Subcase A.1: $d^+_{V(C)}(v_p)>0$} Let then $x_i\in V(C)$ be such that $v_px_i\in A(D)$. It follows from maximality of $C$ that $d^+_{V(P)}(y_{i-1})=0$. In particular, $y_{i-1}u_1\notin A(D)$, and hence $d^+(y_{i-1})+d^-(u_1)\geq a+2$. Therefore
\begin{multline}
\notag
a+2\leq d^+(y_{i-1})+d^-(u_1)=(d^+_{V(C)}(y_{i-1})+d^+_{V(R)}(y_{i-1}))+d^-_{V(P)}(u_1)\\
\leq m+r+p=a\,;
\end{multline}
a contradiction.

\subsection*{Subcase A.2: $d^+_{V(C)}(v_p)=0$} If $v_pu_1\notin A(D)$, then, by assumption,
\[
a+2\leq d^+(v_p)+d^-(u_1)=d^+_{V(P)}(v_p)+d^-_{V(P)}(u_1)\leq 2(p-1)<a\,;
\]
a contradiction. Therefore $v_pu_1\in A(D)$, and so $P$ is, in fact, a cycle. Hence $d^-_{V(R)}(u_i)=0$ and $d^+_{V(R)}(v_j)=0$ for all $u_i,v_j\in V(P)$, by maximality of $P$.

Suppose now that $d^+_{V(C)}(v_j)=0$ for all $v_j\in V(P)$. Then, for any such $v_j$ and $x_i\in V(C)$, we get
\begin{multline}
\notag
a+2\leq d^+(v_j)+d^-(x_i)=d^+_{V(P)}(v_j)+(d^-_{V(C)}(x_i)+d^-_{V(R)}(x_i))\\
\leq p+m+r=a\,;
\end{multline}
a contradiction. Therefore there exist $x_i\in V(C)$ and $v_j\in V(P)$ such that $v_jx_i\in A(D)$. It follows, as in Subcase A.1, that $y_{i-1}u_1\notin A(D)$, and hence
\begin{multline}
\notag
a+2\leq d^+(y_{i-1})+d^-(u_1)=(d^+_{V(C)}(y_{i-1})+d^+_{V(R)}(y_{i-1}))+d^-_{V(P)}(u_1)\\
\leq m+r+p=a\,;
\end{multline}
a contradiction.

\subsection*{Case B: $d^-_{V(C)}(u_1)>0$}

\subsection*{Subcase B.1: $d^+_{V(C)}(v_p)=0$} Let then $y_i\in V(C)$ be such that $y_iu_1\in A(D)$. It follows from maximality of $C$ that $d^-_{V(P)}(x_{i+1})=0$. In particular, $v_px_{i+1}\notin A(D)$, and hence
\begin{multline}
\notag
a+2\leq d^+(v_p)+d^-(x_{i+1})=d^+_{V(P)}(v_p)+(d^-_{V(C)}(x_{i+1})+d^-_{V(R)}(x_{i+1}))\\
\leq p+m+r=a\,;
\end{multline}
a contradiction.

\subsection*{Subcase B.2: $d^+_{V(C)}(v_p)>0$} By Lemma~\ref{lem:4}, $d^+_{V(C)}(v_p)+d^-_{V(C)}(u_1)\leq m-p+1$. If $v_pu_1\notin A(D)$, then
\begin{multline}
\notag
a+2\leq d^+(v_p)+d^-(u_1)=(d^+_{V(C)}(v_p)+d^-_{V(C)}(u_1))+(d^+_{V(P)}(v_p)+d^-_{V(P)}(u_1))\\
\leq(m-p+1)+2(p-1)=m+p-1<a\,;
\end{multline}
a contradiction. Therefore $v_pu_1\in A(D)$, and so $P$ is, in fact, a cycle.

We shall show that $R=\varnothing$ in this case. Suppose otherwise, and let $P'$ be a maximal path in $R$ compatible with $M$. Write $P'=(p_1,\ldots,p_t)$. Then $p_1\in R\cap X$ and $p_t\in R\cap Y$ (see Remark~\ref{rem:1}). Since $P$ is a maximal cycle in $D\setminus V(C)$ compatible with $M$, then $d^-_{V(P)}(p_1)=d^+_{V(P)}(p_t)=0$. Moreover, $d^+_{V(C)}(p_t)+d^-_{V(C)}(p_1)\leq m$, because for every pair of vertices $y_i,x_{i+1}$ on $C$ at most one of the arcs $y_ip_1$ and $p_tx_{i+1}$ exists (by maximality of $C$). Hence
\[
d^+(p_t)+d^-(p_1)=(d^+_{V(C)}(p_t)+d^-_{V(C)}(p_1))+(d^+_{V(R)}(p_t)+d^-_{V(R)}(p_1))\leq m+2r\,,
\]
and so
\begin{multline}
\notag
2a+4\leq(d^+(p_t)+d^-(u_1))+(d^+(v_p)+d^-(p_1))\\
\leq (m+2r)+(d^+_{V(C)}(v_p)+d^-_{V(C)}(u_1))+(d^+_{V(P)}(v_p)+d^-_{V(P)}(u_1))\\
\leq (m+2r)+(m-p+1)+2p=2m+2r+p+1\leq 2m+2r+2p=2a\,;
\end{multline}
a contradiction. We have thus shown that $r=0$, and hence $a=m+p$.

As in the proof of Lemma~\ref{lem:4}, there exist $x_{j_0}$ and $y_{i_0}$ on $C$ such that $y_{i_0}u_1\in A(D)$ and $v_px_{j_0}\in A(D)$. Let $P^{i_0j_0}$ be the path between $y_{i_0}$ and $x_{j_0}$ on $C$, traversed according to the orientation of $C$; of order, say, $2l$. Write $P^{i_0j_0}=(x_{i_0+1},y_{i_0+1},\ldots,x_{j_0-1},y_{j_0-1})$. Then $l\geq p$, because otherwise the cycle
\[
[v_p,x_{j_0},\ldots,y_{i_0},u_1,v_1,\ldots,u_p]
\]
would be strictly longer than $C$.
Further, we can choose the $x_{j_0}$ and $y_{i_0}$ so that
\begin{equation}
\label{eq:7}
y_\nu u_1\notin A(D)\mathrm{\ for\ all\ }y_\nu\in P^{i_0j_0}\quad\mathrm{and}\quad v_px_\nu\notin A(D)\mathrm{\ for\ all\ }x_\nu\in P^{i_0j_0}\,.
\end{equation}
As in the proof of Lemma~\ref{lem:4}, it follows that $d^+_{V(C)}(v_p)+d^-_{V(C)}(u_1)\leq m-l+1$, and hence
\begin{multline}
\label{eq:3}
d^+(v_p)+d^-(u_1)=(d^+_{V(C)}(v_p)+d^-_{V(C)}(u_1))+(d^+_{V(P)}(v_p)+d^-_{V(P)}(u_1))\\
\leq (m-l+1)+2p=a-l+p+1\,.
\end{multline}
By~\eqref{eq:7}, we have $v_px_{i_0+1}\notin A(D)$ and $y_{j_0-1}u_1\notin A(D)$. Hence, and by \eqref{eq:3},
\begin{multline}
\notag
2a+4\leq(d^+(v_p)+d^-(x_{i_0+1}))+(d^+(y_{j_0-1})+d^-(u_1))\\
\leq(d^+(y_{j_0-1})+d^-(x_{i_0+1}))+(a-l+p+1)\,,
\end{multline}
and thus
\begin{equation}
\label{eq:4}
d^+(y_{j_0-1})+d^-(x_{i_0+1})\geq a+l-p+3=m+l+3\,.
\end{equation}
Note that $d^+_{V(P)}(y_{j_0-1})=d^-_{V(P)}(x_{i_0+1})=0$, which follows from the maximality of $C$ and the fact that $P$ is a cycle. Therefore $d^+(y_{j_0-1})=d^+_{V(C)}(y_{j_0-1})$ and $d^-(x_{i_0+1})=d^-_{V(C)}(x_{i_0+1})$, and so, by \eqref{eq:4}, we get
\begin{equation}
\label{eq:5}
d^+_{V(C)}(y_{j_0-1})+d^-_{V(C)}(x_{i_0+1})\geq m+l+3>(m-l-1)+2l+2\,.
\end{equation}
Since $y_{j_0-1}$ and $x_{i_0+1}$ have together at most $2l+2$ neighbours in $V(P^{i_0j_0})\cup\{y_{i_0}\}\cup\{x_{j_0}\}$, then \eqref{eq:5} implies that there exists a pair of vertices $y_s,x_{s+1}$ in $V(C)\setminus(V(P^{i_0j_0})\cup\{y_{i_0}\}\cup\{x_{j_0}\})$ such that $y_sx_{i_0+1}\in A(D)$ and $y_{j_0-1}x_{s+1}\in A(D)$. But then $D$ contains a hamiltonian cycle
\[
[u_1,\ldots,v_p,x_{j_0},\ldots,y_s,x_{i_0+1},\ldots,y_{j_0-1},x_{s+1},\ldots,y_{i_0}]\,.
\]
This contradiction completes the proof of the theorem.
\qed

\begin{remark}
\label{rem:weaker}
Note that the proof of Theorem~\ref{thm:main}, in fact, goes under considerably weaker assumptions. Namely, it suffices to assume that the digraph $D$ contains a complete matching from $X$ to $Y$, and condition $A^*_2$ is satisfied for every pair of vertices $u$ and $v$ such that $u\in X$, $v\in Y$ and $vu\notin A(D)$. That is, we do not need to require any degree condition on pairs of vertices $u$ and $v$ such that $u\in X$, $v\in Y$ and $uv\notin A(D)$. Of course, symmetrically, it suffices to assume a complete matching from $Y$ to $X$ and condition $A^*_2$ satisfied for every pair of vertices $u$ and $v$ such that $u\in X$, $v\in Y$ and $uv\notin A(D)$.
\end{remark}


\bibliographystyle{amsplain}

\begin{thebibliography}{99}

\bibitem{AM} D. Amar and Y. Manoussakis,
 \textit{Cycles and paths of many lengths in bipartite digraphs}, J. Combinatorial Theory Ser. B \textbf{50} (1990), 254--264.

\bibitem{B} C. Berge,
 ``Graphs and hypergraphs'', North-Holland, Amsterdam (1973).

\bibitem{BT} J.C. Bermond and C. Thomassen,
 \textit{Cycles in digraphs - a survey}, J. Graph Theory \textbf{5} (1981), 1--43.

\bibitem{BonT} J.A. Bondy and C. Thomassen,
 \textit{A short proof of Meyniel's theorem}, Discrete Math. \textbf{19} (1977), 195-197.

\bibitem{MM} Y. Manoussakis and I. Milis,
 \textit{A sufficient condition for maximum cycles in bipartite digraphs}, Discrete Math. \textbf{207} (1999), 161--171.

\bibitem{Mey} M. Meyniel,
 \textit{Une condition suffisante d'existence d'un circuit hamiltonien dans un graphe orient{\'e}}, J. Combinatorial Theory Ser. B \textbf{14} (1973), 137-–147. 

\bibitem{NW} C.St.J.A. Nash-Williams,
 \textit{Hamilton circuits in graphs and digraphs}, in ``The many facets of graph theory'', Springer, Lecture Notes \textbf{110} (1969), 237--243.

\bibitem{W} D.R. Woodall, \textit{Sufficient conditions for circuits in graphs}, Proc. London Math. Soc. \textbf{24} (1972), 739--755.

\end{thebibliography}

\end{document}